\documentclass{jgcc} 
\pdfoutput=1
\usepackage{lastpage}
\jgccdoi{16}{1}{7}{15972} 
\jgccheading{}{\pageref{LastPage}}{}{}{Apr.~7,~2025}{Jul.~2,~2025}{}    


\usepackage{hyperref}
\theoremstyle{plain} 

\begin{document}

\title{Andrews-Curtis groups}

\author[R.~H.~Gilman]{Robert H. Gilman}	
\address{Stevens Institute of Technology, Castle Point on Hudson, Hoboken, NJ 07030, USA}	
\email{rgilman@stevens.edu}  

\author[A.~G.~Myasnikov]{Alexei G. Myasnikov}	
\address{Stevens Institute of Technology, Castle Point on Hudson, Hoboken, NJ 07030, USA}
\email{amiasnik@stevens.edu}
\thanks{The second author was supported by Shota Rustaveli National Science Foundation of Georgia under the project FR-21-4713.}





\begin{abstract}
For any group $G$ and integer $k\ge 2$ the Andrews-Curtis transformations act as a permutation group, termed the Andrews-Curtis group $AC_k(G)$, on the subset $N_k(G) \subset G^k$ of all $k$-tuples that generate $G$ as a normal subgroup (provided $N_k(G)$ is non-empty). The famous Andrews-Curtis Conjecture is that if $G$ is free of rank $k$, then $AC_k(G)$  acts transitively on $N_k(G)$.  The set $N_k(G)$ may have a rather complex structure, so it is easier to study the full Andrews-Curtis group $FAC(G)$ generated by AC-transformations on a much simpler set $G^k$. Our goal here is to investigate the natural epimorphism $\lambda\colon FAC_k(G) \to AC_k(G)$. We show that if $G$ is non-elementary torsion-free hyperbolic, then $FAC_k(G)$ acts faithfully on every nontrivial orbit of $G^k$, hence $\lambda\colon FAC_k(G) \to AC_k(G)$ is an isomorphism.
\end{abstract}

\maketitle

\hfill{\small \it In memory of Ben Fine}

\section{The Andrews-Curtis Conjecture}

Andrews-Curtis groups were introduced in connection with the Andrews-Curtis Conjecture (ACC) proposed by James J. Andrews and Morton L. Curtis in 1965 \cite{AC}. According to this conjecture a presentation $P$ which is balanced (the number of its relators equals the number of its generators) presents the trivial group if and only if $P$ can be reduced to the standard presentation of the trivial group by Andrews-Curtis transformations (defined below) of its sequence of relators. In other words $\langle x_1,\ldots, x_k \mid u_1,\ldots, u_k\rangle$ presents the trivial group if and only if $(u_1,\ldots, u_k)$ is AC-equivalent to $(x_1,\ldots, x_k)$.

Many people believe that the ACC is false. Indeed searching for counterexamples is currently an active area of research. In 1985, Akbulut and Kirby \cite{Akbulut} suggested a sequence of potential counterexamples to the ACC of rank 2:
\[
AK(n) =  \langle x,y| ~x^n = y^{n+1}, ~xyx=yxy\rangle, ~n \ge 2.
\]
The presentations $AK(n)$ are balanced  presentations of the trivial group, and it was  conjectured that
that the pair of relators $(x^ny^{-n-1},xyxy^{-1}x^{-1}y^{-1})$ is not AC-equivalent to the pair of generators$(x,y)$.  It turned out later that the presentation $AK(2)$  is AC-trivializable (see \cite{MAD,MMS}), so $AK(2)$ is not a counterexample to the ACC.  The question whether or not the presentations $AK(n)$ with $n >2$ are trivializable is still open despite an ongoing effort by the research community. Currently, $AK(3)$ is the shortest (in the total length of relators) potential counterexample to the ACC.  Indeed, it was proved in \cite{HR} that if $\langle x,y \mid  u= 1, v = 1 \rangle$ is a presentation of the trivial group with  $|u| + |v| \leq 13$ then either
$(u,v) \sim_{AC} (x,y)$
or
$(u,v) \sim_{AC} (x^3y^{-4}, xyxy^{-1}x^{-1}y^{-1})$.
See papers \cite{Metzler,BM,MAD,MMS,PU,P} 
for more details and some particular results.  It was shown in \cite{B} that AC-trivializations could be complex and exponentially long, so brute-force search algorithms  are not going to work easily. Recently, methods of Reinforcement Learning were used with success in search for AC-trivializations of several known balanced presentations of the trivial group, see \cite{Gukov};   as well as axiomatic theorem proving \cite{Lis}.

To construct a counterexample to the ACC one might study the group structure of AC-transformations in an arbitrary group $G$ and apply this knowledge to the ACC.  For example, if the presentation $AK(3)$ is not AC-trivializable in a group $G$, then it is not AC-trivializable in $F_2$, hence the ACC fails. Below we initiate a study of the group of AC-transformations of a  non-elementary torsion-free hyperbolic group.

\section {AC transformations}

AC transformations can be defined for any group $G$. The following are the \emph{elementary Andrews-Curtis transformations} (or AC-moves) on $G^k$, where $k\geq 2$ is a natural number and $G^k$ is the direct power of $k$ copies of $G$.

  \begin{enumerate}
 \item [($R_{ij}$)] $(u_1, \ldots, u_i, \ldots,  u_k) \longrightarrow (u_1,
\ldots,u_iu_j^{\pm 1}, \ldots,  u_k), \ i \ne j;$
  \item [($L_{ij}$)] $(u_1, \ldots, u_i, \ldots, u_k)   \longrightarrow  (u_1,
\ldots,u_j^{\pm 1}u_i, \ldots,  u_k),\   i \ne j;$
 \item [($I_{i}$)] $(u_1, \ldots, u_i, \ldots, u_k)  \longrightarrow  (u_1,\ldots,u_i^{-1}, \ldots, u_k);$
 \item [($C_{i,w}$)] $(u_1, \ldots, u_i,
\ldots, u_k) \longrightarrow (u_1,\ldots,u_i^{w}, \ldots, u_k ), \
w\in G.$
 \end{enumerate}
Transformations  $R_{ij}, L_{ij}, I_i$ are also called \emph{elementary Nielsen transformations}. A composition of finitely many elementary Andrews-Curtis transformations is called  an  \emph{AC-transformation}. Likewise a composition of elementary  Nielsen transformations is a \emph{Nielsen transformation}.
Clearly AC transformations are invertible. The group they generate is $FAC_k(G)$, the \emph{full Andrews-Curtis group} of $G$ of rank $k$.  To define the original \emph{Andrews-Curtis group} $AC_k(G)$, denote by $N_k(G)$ the set of all $k$-tuples in $G^k$ which generate $G$ as a normal subgroup (here we consider only  such $k$ that $N_k(G) \neq \emptyset$). Again, every elementary AC-transformation induces a bijection on $N_k(G)$ and the subgroup of $Sym(N_k(G))$ generated by these bijections is the AC group of $G$, denoted by $AC_k(G)$. The AC conjecture can be stated as follows: for every free group $F$ of rank $k \geq 2$ the Andrews-Curtis group $AC_k(F)$ acts transitively on the set $N_k(F)$.  The set $N_k(G)$ can be quite complex, for example, it is known that there is no algorithm to decide whether a group given by a finite presentation is trivial or not \cite{Adjan,Rabin}. Whether such an algorithm exists for balanced presentations is an open and difficult problem (see \cite{BMS}). It follows that the set $N_k(F)$ is computably enumerable, but it is not known if it is computable. On the other hand, if $G$ is finitely generated with decidable word problem then the set $G^k$ is computable. In particular, the set $F^k$ is computable. So it might be easier to study the group $FAC_k(G)$ then $AC_k(G)$.  Observe, that the restriction of a bijection $\alpha \in FAC_k(G)$ onto the set $N_k(G)$ gives  a bijection $\bar \alpha \in AC_k(G)$. It is easy to see that the map $\alpha \to \bar \alpha$ gives rise to  an epimorphism
$$
\lambda_{G,k}\colon FAC_k(G) \to AC_k(G).
$$

\medskip
\noindent
{\bf General Problem 1.} {\it 
For a given group $G$ and $k\geq 2$, describe the kernel of the epimorphism
$\lambda_{G,k}\colon FAC_k(G) \to AC_k(G)$.}

\medskip

Below we address this problem for torsion-free hyperbolic groups $G$.

The main result of the paper (proved in Section~\ref{sec:ACG}) is as follows.

\begin{thm}\label{thm:main} Let $G$ be a torsion-free non-elementary hyperbolic group. Then $FAC_k(G)$ acts faithfully on every nontrivial orbit in $G^k$ (the trivial orbit consists of the $k$-tuple $(1,1,\ldots, 1)$).
\end{thm}

\begin{cor}
    Let $G$ be a non-elementary torsion-free hyperbolic group. Then  for any $k \geq 2$ $\lambda_{G,k}\colon FAC_k(G) \to AC_k(G)$ is an isomorphism.
\end{cor}

Note that in \cite{Rom} Roman'kov independently proved that $\lambda_{G,k}$ is an isomorphism for free nonabelian groups $G = F_k$ of rank $k\geq 2$.
 
\section{Equations}

We require some results about equations over hyperbolic groups. Throughout this section $G$ stands for a non-elementary torsion free hyperbolic group; for example a finitely generated nonabelian free group. Standard references for background on hyperbolic groups are~\cite{G,ABC, CDP, GH}.

We state some well-known properties of $G$ as a lemma.
\begin{lem}\label{lem:hyp} Let $G$ be a torsion-free non-elementary hyperbolic group.
\begin{enumerate}
\item Let $H$ be the centralizer in $G$ of a non-identity elements. Then $H$ is cyclic and malnormal i.e., $H^x\cap H = 1$ if $x \in G-H$.
\item If $F$ is free of finite rank, then the free product $G*F$ is non-elementary torsion-free hyperbolic.
\item $G$ satisfies the big powers property~\cite{O}. In other words for any sequence $v_1,\ldots, v_n \in G$ such that $v_i$ does not commute with $v_{i+1}$ for $1\le i < n$, there exists an integer $m$ such that
\[ v_1^{r_1} \cdots v_n^{r_n} \neq 1 \mbox{ whenever $r_i\ge m$ for all $i$}.\]
\end{enumerate}
\end{lem}

Now we consider equations over $G$.

\begin{defi} An equation $E(x_1, \ldots, x_m)$ over $G$ is an element of the free product $G*X$ where $X$ is freely generated by $\{x_1,\ldots, x_m\}$. A tuple $g_1,\ldots g_m$ of elements of $G$ is a solution to $E$ if the unique homomorphism $G*X\to G$ which is the identity on $G$ and sends each $x_i$ to $g_i$ maps $E$ to the identity.
\end{defi}

Each equation $E(x_1, \ldots, x_m)$ may be written as E = $a_0 x_{i_1}^{d_1} a_1 \cdots x_{i_n}^{d_n}a_n$, where the $a_j$'s are elements of $G$, and the $d_j$'s are nonzero integers. Without loss of generality we may assume that if $a_j = 1$ for some $j$ with $0< j < n$ then $x_{i_j} \neq x_{i_{j+1}}$. In addition since conjugation of an equation does not change the solution set, we may suppose $a_0=1$. Thus we have
\begin{equation}\label{eq1}
E = x_{i_1}^{d_1} a_1 x_{i_2}^{d_2} a_2\cdots x_{i_n}^{d_n}a_n \mbox{ with $x_{i_j} \neq x_{i_{j+1}}$ if $a_j=1$.}
\end{equation}

\begin{thm}\label{th:trivial}
Let $E(x_1, \ldots, x_m)$ be an equation over a non-elementary torsion-free word-hyperbolic group $G$. If all $m$-tuples in $G$ are solutions to $E$, then  $E=1$ in $G*X$.
\end{thm}
\begin{proof}
We argue by induction on $n$. The cases $n=0,1$ are left to the reader, so we may assume $n\ge 2$.

By our hypothesis, substituting $1$ for all the $x_j$'s yields $a_1\cdots a_n= 1$ in $G$. It follows that we may rewrite Equation~(\ref{eq1}) as a product of conjugates of powers.
\begin{equation}\label{eq2}
E = x_{i_1}^{d_1} a_1x_{i_2}^{d_2}a_1^{-1} a_1a_2x_{i_3}^{d_3}(a_1a_2)^{-1}\cdots
(a_1\cdots a_{n-1})x_{i_n}^{d_n}(a_1\cdots a_{n-1})^{-1}
\end{equation}
Define $u_1(x_{i_1})=x_{i_1}^{d_1}$, $u_2(x_{i_2})=a_1x_{i_2}^{d_2}a_1^{-1} $u
$u_3(x_{i_3})=a_1a_2x_{i_3}^{d_3}(a_1a_2)^{-1} $, etc. The group $G$ satisfies the big powers condition (see Lemma~\ref{lem:hyp}) hence for substitutions of elements $g_{i_j}^{r_j}\in G$ for $x_{i_j}$ either
\[E( g_{i_1}^{r_1}, g_{i_2}^{r_2}, \ldots, g_{i_n}^{r_n})\neq 1
\]
for all  sufficiently large integers $r_j$ or for some $ j < n$, $u_j(g_{i_j})$ and $u_{j+1}(g_{i_{j+1}})$ commute. In the first case there are many tuples $g_{i_1}^{r_1}, g_{i_2}^{r_2}, \ldots, g_{i_n}^{r_n}$ which are not solutions to our equation $E$, so we may assume that for some $ j < n$, $u_j(g_{i_j}^{r_j})$ and $u_{j+1}(g_{i_{j+1}}^{r_{j+1}})$ commute.

In other words
\[(a_1\cdots a_{j-1})g_{i_j}^{r_jd_j}(a_1\cdots a_{j-1})^{-1}\]
commutes with
\[(a_1\cdots a_{j})g_{i_{j+1}}^{r_{j+1}d_{j+1}}(a_1\cdots a_{j})^{-1}\]
whence $g_{i_j}^{r_jd_j}$ commutes with
$a_{j}g_{i_{j+1}}^{r_{j+1}d_{j+1}}a_{j}^{-1}$.
By Lemma~\ref{lem:hyp}, the group $G$ is commutative transitive (since all proper centralizers are commutative) hence $g_{i_j}$ and $a_jg_{i_{j+1}}a_j^{-1}$ commute. But clearly we can choose $g_{i_j} \in G$ such that $g_{i_j}$ and $a_jg_{i_{j+1}}a_j^{-1}$ do not commute.  Thus there is a substitution $x_i \to g_i^{r_i}$ which does not yield $1$ in the equation $E =1$.
 \end{proof}

\begin{rem}
    Theorem \ref{th:trivial}  shows that, in terms of algebraic geometry over groups, the radical of the affine space $G^n$ for non-elementary torsion-free hyperbolic groups $G$ is trivial.
\end{rem}

\begin{rem}
    The proof of Theorem \ref{th:trivial} shows that the result holds for any group $G$ that satisfy the following conditions:
    \begin{enumerate}
        \item [1)] $G$ is CSA, i.e., centralizers of non-trivial elements are abelian and malnormal; there are many such groups which are not hyperbolic  (see, for example,  \cite{MR,GKM,GL})
        \item [2)] $G$ satisfies the  big powers condition (see examples in \cite{O,KMS});
        \item [3)] For any finite subset of non-trivial elements $A \subseteq G$,  there is an element $g \in G$ such that $[a,g] \neq 1$ for every $a\in A$.
    \end{enumerate}
\end{rem}

\section{Proof of Theorem~\ref{thm:main}}\label{sec:ACG}

{\bf Notation.} When the value of $k$ is irrelevant, we will write $FAC(G)$ in place of $FAC_k(G)$. In addition we will abbreviate $(u_1,\ldots,u_k)$ to $\vec u$.

Notice that for each $\alpha\in FAC(G)$,
$\alpha(\vec u)$ is computed by a fixed sequence of elementary $AC$-moves on $\vec u$. Performing the same sequence on a tuple of indeterminants $\vec x =(x_1,\ldots, x_k)$ yields group words $(W_1, \ldots, W_k)$ in the free product $G*X$. Here $X$ is the free group over $\{x_1,\ldots, x_k\}$.  We record this observation as a lemma.

\begin{lem}\label{lem:Nielsen}
For every $\alpha\in FAC(G)$ there are group words
$W_i(x_1,\ldots, x_k)$ for $1\le i \le k$ over indeterminants
$x_1,\ldots, x_k$, such that
\[\alpha(u_1,\ldots, u_k) = (W_1(u_1,\ldots, u_k ), \ldots, W_k(u_1,\ldots, u_k)).\]\end{lem}
Let $G$ be a torsion-free non-elementary hyperbolic group, and suppose $\alpha \in FAC(G)$ fixes all elements in the orbit of $\vec u = (u_1,\ldots, u_k)$. Without loss of generality we may assume $u_i\ne 1$ for all $i$. It suffices to show that if $\alpha \in FAC(G)$ fixes all conjugates of $\vec u$, i.e., all sequences $(u_1^{h_1}, \ldots, u_k^{h_k})$ as the $h_i$'s run over all elements of $G$, then $\alpha = 1$.

By Lemma~\ref{lem:Nielsen} there are group words $W_i(x_1,\ldots, x_k)$ over $G*X$ such that
\[\alpha(v_1,\ldots, v_k) = (W_1(v_1,\ldots, v_k ), \ldots, W_k(v_1,\ldots, v_k))\]
for all $\vec v\in G^k$. By hypothesis $W_i(u_1^{d_1}, \ldots, u_k^{d_k})= u_i^{d_i}$
for all $d_1,\ldots, d_k\in G$.

Since $G*X$ is non-elementary torsion-free hyperbolic, it follows from Theorem~\ref{th:trivial} that
\begin{equation}\label{eq:3} W_i(u_1^{x_1},\ldots, u_k^{x_k}) = u_i^{x_i}
\mbox{ in $G*X$}.
\end{equation}
By properties of free products, there is an endomomorphism $f\colon G*X\to G*X$ which is the identity on $G$ and maps each $x_i$ to $u_i^{x_i}$.  A straightforward argument shows that $f$ is injective. Hence Equation~\ref{eq:3} implies $W_i(x_1,\ldots, x_k) =x_i$. Thus $\alpha = 1$ as desired.

\section{Open Problems}

Here we collect some open problems on AC-groups $AC_k(G)$. Some of them are new, others appear in various presentations, preprints, or papers.

\medskip
\noindent
{\bf General Problem 2.} {\it Study groups $ FAC_k(G)$ and $AC_k(G)$ for different platform groups $G$. 
}

\medskip

More particular problems are listed below.
 \begin{prob}
 Which groups $AC(G)$ are finitely presentable?
 \end{prob}
  Note, that in \cite{Rom} Roman'kov showed that the group $AC_2(F_2)$, where $F_2$ is a free group of rank 2, is not finitely presented. This brings the following question for free groups $F_k$ of rank $k$.
  \begin{prob}
      Is it true that the groups $AC_k(F_k)$ are finitely presented for $k \geq 3$?
  \end{prob}
 \begin{prob}
 Does the conclusion of Theorem~\ref{thm:main} hold for partially commutative groups?
 \end{prob}
  \begin{prob}\label{prob:normal}
 Find ``good" (quasi-geodesic) normal forms of elements in $AC_k(F_k)$
 \end{prob}
 A solution to Problem~\ref{prob:normal} will enhance efficacy of search for counterexamples to the ACC.s
 \begin{prob}
 For which $k$ does the group $AC_k(F_k)$ have Kazhdan property (T)?
 \end{prob}
Positive solution to this problem would explain why the analog of the product replacement algorithm for generators of normal subgroups in black-box groups works rather well (see \cite{BKM}).

\bibliographystyle{plain}
\bibliography{refs}
\end{document}